\documentclass{amsart} \usepackage{bbm}
\usepackage[top=2.8cm,left=2.8cm,right=2.8cm,bottom=2.6cm]{geometry}
\usepackage{amsfonts}
\usepackage{amsmath,amsthm,amscd,mathrsfs,amssymb,graphicx,enumerate,
  dsfont}
\usepackage{xypic}

\newtheorem{thm}{Theorem}[section]
\newtheorem{cor}[thm]{Corollary}
\newtheorem{lem}[thm]{Lemma}
\newtheorem{prop}[thm]{Proposition}
\newtheorem{conj}[thm]{Conjecture}
\theoremstyle{definition}
\newtheorem{defn}[thm]{Definition}

\newtheorem{rmk}[thm]{Remark}

\DeclareMathOperator{\Hom}{Hom}

\newcommand{\C}{\ensuremath\mathds{C}}
\newcommand{\R}{\ensuremath\mathds{R}}
\newcommand{\Z}{\ensuremath\mathds{Z}}

\newcommand{\PP}{\ensuremath\mathds{P}}
\newcommand{\calO}{\ensuremath\mathcal{O}}

\newcommand{\HH}{\ensuremath\mathrm{H}}

\newcommand{\Set}[2]{\left\{#1 :#2\right\}}

\begin{document}

\title{Vanishing cycles under base change and the integral Hodge conjecture}
\author{Mingmin Shen}

\thanks{2010 {\em Mathematics Subject Classification.} }


\address{
KdV Institute for Mathematics, University of Amsterdam, P.O.Box 94248, 1090 GE Amsterdam, Netherlands}
\email{M.Shen@uva.nl}

\date{\today}

\begin{abstract} 
In this paper we discuss an obstruction to the integral Hodge conjecture, which arises from certain behavior of vanishing cycles. This allows us to construct new counter-examples to the integral Hodge conjecture. One typical such counter-example is the product of a very general hypersurface of odd dimension and an Enriques surface. Our approach generalizes the degeneration argument of Benoist--Ottem \cite{bo}.
\end{abstract}

\maketitle

\section{Introduction}

In this paper, we work over the field $\C$ of complex numbers. Let $X$ be a smooth projective variety, then the cohomology group of $X$ carries a Hodge structure given by
\[
\HH^k(X,\Z)\otimes \C = \bigoplus_{p+q=k} \HH^{p,q}(X), \qquad \overline{\HH^{p,q}(X)} = \HH^{q,p}(X). 
\]
The group of integral Hodge classes, denoted $\mathrm{Hdg}^{2p}(X,\Z)$, consists of all the elements $\alpha\in \HH^{2p}(X,\Z)$ such that $\alpha\otimes 1 \in \HH^{2p}(X,\Z)\otimes \C$ is in the summand $\HH^{p,p}(X)$. One easily sees that the torsion classes are all integral Hodge classes, \textit{i.e.}
\[
 \HH^{2p}(X,\Z)_{\mathrm{tor}} \subseteq \mathrm{Hdg}^{2p}(X,\Z).
\]
W.~Hodge discovered that the cohomology class $[Z]$ of an algebraic cycle $Z$ on $X$ is always an integral Hodge class.

\begin{conj}[Integral Hodge Conjecture]
Every integral Hodge class is the cohomology class of an algebraic cycle.
\end{conj}

It is known since Atiyah--Hirzebruch \cite{ah} that the integral Hodge conjecture is false. Since then, many theories and techniques were developed to construct more counter-examples. In the recent paper \cite{bo}, Benoist and Ottem used a degeneration argument to show that certain integral Hodge class is not algebraic. In this paper, we generalize their method to produce more counter-examples.

Our method is based on the following simple observation. Let $Y$ be a smooth projective variety and $U\subseteq Y$ a dense open subvariety. If $Z_U$ is an algebraic cycle on $U$, then it extends to an algebraic cycle $Z$ on $Y$ by taking the closure. However, a (locally finite) toplological cycle $z_U$ on $U$ does not necessarily extend to one on $Y$. The main reason is that the closure of $z_U$ might have a nontrivial boundary. Hence a cohomology class being algebraic imposes stronger extension property on the class. We make the following definition to make the discussion easier.

\begin{defn}
Let $\pi: \mathcal{X}\rightarrow B$ be a flat projective morphism between smooth complete varieties. Let $0\in B$ be a closed point such that $X=\mathcal{X}_0:=\pi^{-1}0$ is a smooth fiber. Let $\alpha\in \HH^{k}(X,R)$ be a cohomology class with coefficients in a commutative ring $R$. We say that $\alpha$ is \textit{extendable} if the following hold. 
\begin{itemize}
\item There exists a smooth complete variety $\tilde{B}$ together with a generically finite morphism $\tilde{B}\rightarrow B$.
\item For some resolution $\tilde{\mathcal{X}}$ of $\mathcal{X}':=\mathcal{X}\times_B\tilde{B}$ and some preimage $\tilde{0}\in\tilde{B}$ of $0$, we have $X = \tilde{\mathcal{X}}_{\tilde{0}}:=\tilde\pi^{-1}(\tilde{0})$ where $\tilde{\pi}: \tilde{\mathcal{X}} \rightarrow \tilde{B}$ is the morphism induced by $\pi$.
\item There exists a cohomology class $\tilde{\alpha}\in \HH^k(\tilde{\mathcal{X}}, R)$ such that $\alpha = \tilde{\alpha}|_X$.
\end{itemize} 
\end{defn}

\begin{rmk}
Assume that $0\in B$ is a very general point. If $\alpha \in \HH^{2p}(X,\Z)$ is the class of an algebraic cycle $Z$, then $\alpha$ is extendable. Indeed, one can identify $X=\pi^{-1}0$ with the geometric generic fiber $\mathcal{X}_{\bar\eta}$. The algebraic cycle $Z$ can then be defined over a finite extension of $\eta_B$. Then a standard argument shows that there exists some generically finite morphism $\tilde{B}\rightarrow B$ and an algebraic cycle $\mathcal{Z}'$ on ${\mathcal{X}}':=\mathcal{X}\times_B\tilde{B}$ such that $[\mathcal{Z}'_{\tilde{0}}] = \alpha$. The strict transform $\tilde{\mathcal{Z}}$ of $\mathcal{Z}'$ in the resolution $\tilde{\mathcal{X}}$ sastisfies $[\tilde{\mathcal{Z}}_{\tilde{0}}] = \alpha$. Thus we can simply take $\tilde{\alpha} = [\tilde{\mathcal{Z}}]$.
\end{rmk}

This remark gives rise to the following non-algebraicity criterion.

\vspace{3mm}
\noindent \textbf{(Non-algebraicity criterion)} \textit{If an integral Hodge class $\alpha$ is not extendable, then this class is not algebraic on a very general fiber.}
\vspace{3mm}

Our first main result is the following non-extendability of vanishing cycles on odd dimensional smooth hypersurfaces.

\begin{thm}[Theorem \ref{thm non-ext}]
Let $\mathcal{X}\longrightarrow B=\PP^1$ be a Lefschetz pencil of smooth hypersurfaces of odd dimension $n$. Let $X=\pi^{-1}0$ be a smooth fiber. Then every non-zero element $\alpha\in \HH^n(X,R)$ is non-extendable, where $R$ is a nonzero commutative ring.
\end{thm}

This non-extendability can be used to obstruct algebraicity as follows. For simplicity, we take $S$ to be an Enriques surface. Then $\HH^3(S,\Z)=\Z/2\Z$ with a generator $u$. 

\begin{cor}[Corollary \ref{cor non-alg}]
Let $X\subset \PP^{n+2}$ be a very general hypersurface of odd dimension $n$. For every element $\alpha\in \HH^n(X,\Z)$ which is not divisible by $2$, the torsion class $\alpha\otimes u\in \HH^{n+3}(X\times S,\Z)$ is not extendable (in a Lefschetz pencil) and hence not algebraic.
\end{cor}

\begin{rmk}
The proof of the corollary reduces to the non-extendability of the image $\bar\alpha$ of $\alpha$ in $\HH^n(X,\Z/2\Z)$; see section 3. In \cite{bo}, Benoist and Ottem considered the case where $X=E$ is a very general elliptic curve. Their method involves an element $\alpha\in \HH^1(E,\Z)$. Instead of considering the topological extendability of $\bar\alpha \in \HH^1(E,\Z/2\Z)$, they consider the degeneration of the double cover $E'\longrightarrow E$ associated to $\bar\alpha$. The obstruction in the Benoist--Ottem example was given an interpretation via unramified cohomology by Colliot-Th\'el\`ene \cite{ct}. It is interesting to see if a similiar interpretation exist for our generalisation.

The counter-examples to the integral Hodge conjecture obtained via the above corollary are all around the range of middle degree cohomology. 
\end{rmk}

Our method also works when $X$ is a hyperplane section of a smooth projective variety $Y$. This more general case is treated in Theorem \ref{thm general case}. Our result shows that there exist integral Hodge classes which are not extendable. Given the outstanding Hodge conjecture, it is natural to ask whether every rational Hodge class on a very general fiber is extendable. 

\vspace{3mm}

\noindent\textbf{Acknowledgement.} A large part of the computations in Section 2 were carried out in the summer of 2018 when I was visiting University of Science and Technology of China. I thank Mao Sheng for the invitation. I also thank John Ottem for the interesting discussions related to this paper. This research was partially supported by NWO Innovational Research Incentives Scheme 016.Vidi.189.015.

\section{Vanishing cycles under blow-up}

\subsection{An induction process}

Let $r\geq 2$ be an integer and let $X_r$ be a complex analytic space with a unique singular point $P_r$. Assume that $P_r$ has an open neighborhood $U_r$ such that
\[
 U_r \cong \{(t,\mathbf{z})\in \C\times \C^{n+1}: t^r=z_0^2 + \cdots + z_n^2\}.
\]
Let $D_\epsilon^{n+1}:=\{\mathbf{x}\in\mathds{R}^{n+1}: |\mathbf{x}|\leq \epsilon\}$ be the closed disc. We have continuous maps
\begin{equation}\label{eq varphi_r}
\varphi_{r,a} : D^{n+1}_{\epsilon} \longrightarrow U_r, \qquad \varphi_{r,a}(\mathbf{x}) = (\xi_r^a|\mathbf{x}|^{2/r},\mathbf{x}),
\end{equation}
where $\xi_r=\exp(\frac{2\pi i}{r})$ and $a=0,1,\ldots, r-1$. Let $M'$ be the blow-up of $M=\C\times \C^{n+1}$ at the point $P=(0, \mathbf{0})$. Let $U'_r\subset M'$ be the strict transform of $U_r$ and $\rho: X'_r \rightarrow X_r$ be the resulting blow-up of $X_r$ at the point $P_r$. We write
\[
 M\backslash \{(0,\mathbf{0})\}= V \cup V_0\cup \cdots \cup V_n, \quad V=\{t\neq 0\},\;\; V_i = \{z_i\neq 0 \}.
\]
Then $M'$ admits a corresponding open cover
\[
 M' = V' \cup V'_0 \cup \cdots \cup V'_n.
\]
Here $V' \cong \C\times \C^{n+1}$ and the map $V' \rightarrow V\cup\{(0,\mathbf{0})\}$ is given by
\[
 (t,w_0,\ldots,w_n) \mapsto (t, tw_0,\ldots,tw_n).
\]
Similarly, we have $V'_i\cong \C\times \C^{n+1}$ and the map $V'_i\rightarrow V_i\cup\{(0,\mathbf{0})\}$ is given by
\[
(t,w_0,\ldots, w_n) \mapsto (tw_i, w_0w_i, \ldots, w_{i-1}w_i, w_i, w_{i+1}w_i, \ldots, w_nw_i).
\]
The exceptional divisor $E$ of the blow-up $M'\rightarrow M$ is isomorphic to $\PP^{n+1}$ and the open cover
\[
 E= (E\cap V')\cup\bigcup_{i=0}^n (E\cap V'_i)
\]
is the standard affine cover associated to the homogeneous coordinates $[T:Z_0:\cdots : Z_n]$ of $\PP^{n+1}$.

We have the following commutative diagram
\[
\xymatrix{
 U'_r\ar[r]\ar[d]_{\rho} &M'\ar[d]\\
 U_r\ar[r] &M
}
\]
Furthermore, $U'_r\cap V' \subset V'$ is defined by the equation
\[
 t^{r-2} = w_0^2 + \cdots +w_n^2.
\]
Thus $U'_r\cap V'$ is smooth if $r=2,3$; it is singular at the point $(t,\mathbf{w}) = (0,\mathbf{0})$ if $r\geq 4$. The intersection $U'_r\cap V'_i$ is defined by the equation
\[
t^rw_i^{r-2} = w_0^2+\cdots + w_{i-1}^2 + 1 + w_{i+1}^2 +\cdots + w_n^2
\]
which is always smooth. 

If $r=2$, then the exceptional divisor of $U'_r\rightarrow U_r$ is the smooth quadric
\[
 Q=\{T^2 = Z_0^2+\cdots Z_n^2\} \subset E=\PP^{n+1}.
\]
If $r\geq 3$, then the exceptional divisor of $U'_r\rightarrow U_r$ is the singular quadric
\[
Q'=\{0=Z_0^2 + \cdots + Z_n^2\}\subset E=\PP^{n+1}.
\]
The singular point of $Q'$ is $P'_r:=[1:0:\cdots:0]\in\PP^{n+1}$.

The map $\varphi_{r,a}$ restricted to $D^{n+1}_\epsilon\backslash\{\mathbf{0}\}$ lifts to $V'$, which is given by
\[
\mathbf{x}=(x_0,\ldots, x_n) \mapsto (\xi_r^a|\mathbf{x}|^{2/r}, \xi_r^{-a}|\mathbf{x}|^{1-2/r}\theta(\mathbf{x})), \qquad\mathbf{x}\in D_{\epsilon}^{n+1}\backslash \{\mathbf{0}\},
\]
where $\theta(\mathbf{x})= \frac{\mathbf{x}}{|\mathbf{x}|}\in S^{n}$. If $r\geq 3$, then the above map extends to 
\[
 \varphi'_{r,a}: D^{n+1}_{\epsilon} \longrightarrow V'
\]
by the same formula and $\mathbf{0} \mapsto (0,\mathbf{0})\in V'$, which is the singular point of $Q'$. In this case, $X'_r$ is locally defined by the equation
\[
 t^{r-2} = w_0^2 + \cdots + w_n^2.
\]
The following lemma implies that the same argument can be repeated on $(X_{r-2}, P_{r-2}) =(X'_r,P'_r)$.
\begin{lem}\label{lem induction}
(1) If $r=2$, then the lifting of $\varphi_{2,a}|_{D^{n+1}_\epsilon\backslash\{\mathbf{0}\}}$ to $X'_r$ can be extended to a continuous map
\[
\varphi'_{2,a} : [0,\epsilon]\times S^n \longrightarrow X'_r
\]
such that $\varphi'_{2,a}(\rho,\mathbf{x}) = \varphi_{2,a}(\rho\mathbf{x})$ for all $(\rho,\mathbf{x})\in (0,\epsilon]\times S^n$. Furthermore, 
\[
\varphi'_{2,a}(0,\mathbf{x}) = [1:(-1)^ax_0:\cdots:(-1)^ax_n]\in Q,
\]
which is an $n$-sphere in $Q\subset E\cong\PP^{n+1}$ that vanishes in the homology of $\PP^{n+1}$. In this case, $X'_r$ is smooth. 

(2) If $r = 3$, then the lifting of $\varphi_{r,a}|_{D^{n+1}_\epsilon\backslash\{\mathbf{0}\}}$ to $X'_r$ can be extended to a continuous map
\[
\varphi'_{r,a} : D^{n+1}_\epsilon \longrightarrow X'_r
\]
such that $\varphi'_{r,a}(\mathbf{0}) = P'_r$ is the singular point of $Q'$. In this case, $X'_r$ is smooth.

(3) If $r\geq 4$, then $X'_r$ is singular at the point $P'_r$ where $X'_r$ is locally defined by an equation
\[
{t'}^{r-2} = {z'}_0^2 + {z'}_1^2+\cdots + {z'}_n^2.
\]
The lifting of $\varphi_{r,a}|_{D^{n+1}_\epsilon\backslash\{\mathbf{0}\}}$ to $X'_r$ can be extended to a continuous map
\[
\varphi_{r-2,0}: D^{n+1}_{\epsilon'} \longrightarrow X'_r, 
\]
with $\varphi_{r-2,0}(\mathbf{0}) = P'_r$ being the singular point of $X'_r$ and $\varphi_{r-2,0}(\mathbf{x}) = (|\mathbf{x}|^{\frac{2}{r-2}}, \mathbf{x})$ for general $\mathbf{x}$. 
\end{lem}
\begin{proof}
For (1), we note that, in this case, the lifting of $\phi_{2,a}$ restricted to $D^{n+1}_{\epsilon}\backslash\{\mathbf{0}\}$ is given by
\[
 \mathbf{x}=(x_0,\ldots,x_n) \mapsto ((-1)^a|\mathbf{x}|, (-1)^a\theta(\mathbf{x})).
\] 
Note that $(0,\epsilon]\times S^n\cong D^{n+1}_{\epsilon}\backslash\{\mathbf{0}\}$. It is clear that the above map extends to a continuous map $\varphi'_{2,a}: [0,\epsilon]\times S^n$ as stated.

Statement (2) can be shown similarly.

We show the last statement and assume that $r\geq 4$. We have already seen that $X'_r$ has a unique singular point $P'_r$ such that $X'_r$ is locally defined by
\[
t^{r-2} = w_0^2 + w_1^2 + \cdots + w_n^2
\]
and that there is a lifting $\varphi'_{r,a}$ of $\varphi_{r,a}$ given by
\[
 \varphi'_{r,a}(\mathbf{x}) =   (\xi_r^a|\mathbf{x}|^{2/r}, \xi_r^{-a}|\mathbf{x}|^{1-2/r}\theta(\mathbf{x})), \quad \mathbf{x}\in D^{n+1}_\epsilon\backslash \{\mathbf{0}\}
\]
and $\varphi'_{r,a}(\mathbf{0}) = (0,\mathbf{0})=P'_r$. We introduce a new set of coordinates
\begin{align*}
 t' &= \xi_r^{-a}t,\\
 z'_i & = \xi_r^a w_i
\end{align*}
and we see that the local defining equation of $X'_r$ around $P'_r$ becomes
\[
{t'}^{r-2} = {z'}_0^2 + {z'}_1^2+\cdots + {z'}_n^2.
\]
Furthermore, in terms of $(t',\mathbf{z}')$, the map $\varphi'_{r,a}$ becomes
\[
\mathbf{x}\mapsto (|\mathbf{x}|^{2/r},|\mathbf{x}|^{1-2/r}\theta(\mathbf{x}))
\]
for $\mathbf{x}\neq\mathbf{0}$ and $\mathbf{0}\mapsto (0,\mathbf{0})$. Let $\epsilon' = \epsilon^{1-2/r}$ and define a homeomorphism $D^{n+1}_\epsilon \rightarrow D^{n+1}_{\epsilon'}$ by $\mathbf{x}\mapsto \mathbf{x}' = |\mathbf{x}|^{-2/r}\mathbf{x}$ for $\mathbf{x}\neq \mathbf{0}$ and $\mathbf{0}\mapsto \mathbf{x}'=\mathbf{0}$. It follows that the composition $D^{n+1}_{\epsilon'}\rightarrow D^{n+1}_{\epsilon}\rightarrow X'_r$ becomes
\[
\mathbf{x}' \mapsto (|\mathbf{x}'|^{\frac{2}{r-2}}, \mathbf{x}').
\]
This concludes the proof.
\end{proof}

\subsection{Application to vanishing cycles}

\subsection{Local situation}
Let $\Delta\subset\C$ be the unit open disc in the complex plane and $\Delta^*=\Delta\backslash \{0\}$. Let $\pi: X\rightarrow \Delta$ be a proper map of complex manifolds such that $X^*\rightarrow \Delta^*$ is smooth, where $X^*=\pi^{-1}\Delta^*$. We write $X_t:=\pi^{-1}t$, $t\in \Delta$. Assume that $X_0=\pi^{-1}(0)$ has one ordinary double point $P$ such that we have local coordinates $(z_0,\ldots, z_n)$ on an open neighborhood $U$ of $P$ and 
\[
\pi(z_0,z_1,\ldots, z_n) = z_0^2 + z_1^2 + \cdots + z_n^2.
\]
Let $\psi_r:\Delta \rightarrow \Delta$ be the map $\psi_r(t) = t^r$ and $X_r:= \psi_r^* X$ be the base change of $X$, where $r\geq 2$. Namely, we have the following fiber product quare
\[
\xymatrix{
 X_r\ar[r]^{\psi'_r}\ar[d]_{\pi_r} &X\ar[d]^{\pi}\\
 \Delta\ar[r]^{\psi_r} &\Delta
}
\]
Let $U_r=\psi_r^* U$ be the corresponding base change of $U$. Thus $U_r$ is an open neighborhood of the point $P_r={\psi'_r}^{-1}(P)$. Hence we have $U_r\subset \C\times \C^{n+1}$ defined be
\[
U_r= \Set{(t,z_0,\ldots,z_n)\in \C\times \C^{n+1}}{t^r = z_0^2 + z_1^2 + \cdots + z_n^2}.
\]
Then $P_r$ is the unique singular point of $U_r$ and it has coordinates $(0,0,\ldots,0)$. 

For any positive real number $\epsilon \in (0,1)$, let 
\[
S_{\epsilon}^{n}=\{(\epsilon^2, x_0,\ldots, x_n)\in U_r: x_0^2+ \cdots + x_n^2=\epsilon^2, x_i\in \mathds{R}  \}\subset X_{\epsilon^2}
\]
be a vanishing sphere. Let
\[
D_{\epsilon}^{n+1} \rightarrow {X}, \quad (x_0,\ldots,x_n)\mapsto (\sum x_i^2, x_0, \ldots, x_n)
\]
where $D_{\epsilon}^{n+1}=\{(x_0,\ldots,x_n)\in \R^{n+1}: x_0^2+\cdots + x_n^2\leq \epsilon^2\}$ is a small disc whose boundary gives the vanishing shpere $S_{\epsilon}^n$.

Let $\rho: X'_{r}\rightarrow X_r$ be the blow-up of $X_r$ at the point $P_r$ and let $U'_r := \rho^{-1} U_r$. There are $r$ different ways to lift the map $D^{n+1}_{\epsilon}\backslash \{0\}\rightarrow {X}^*$ to $(X'_r)^*$ given by
\begin{equation}
 \mathbf{x}=(x_0,\ldots,x_n) \mapsto (\xi_r^a|\mathbf{x}|^{2/r}, x_0,\ldots, x_n),\quad a=0,1,\ldots,r-1.
\end{equation}
We have seen in Lemma \ref{lem induction} that $X'_r$ is again singular if $r\geq 4$ with a single singular point $P'_r$ and the blow-up proess can be repeated. 

\begin{lem}\label{lem local vc}
The following statements are true.

(1) The singularity of $X_r$ can be resolved by successively blowing up the singular points
\[
\tilde{X} = X_r^{(b)}\longrightarrow \cdots \longrightarrow X_r^{(2)}\longrightarrow X_r^{(1)}=X'_r\longrightarrow X_r
\]
where $b=[\frac{r}{2}]$ and $X^{(b)}_r$ is smooth.

(2) Let $Q\subseteq \tilde{X}$ be the exceptional divisor of the last blow-up $X^{(b)}_r\longrightarrow X_r^{(b-1)}$. Then $Q$ is a component of $\tilde{X}_0 = \tilde{\pi}^{-1}(0)$, where $\tilde{\pi}:\tilde{X}\longrightarrow \Delta$ is the composition of all the blow-ups together with $\pi_r$. If $r=2b$ is even, then $Q$ is a smooth quardric hypersurface of dimension $n$; if $r=2b+1$ is odd, then $Q$ is a cone over a smooth quadric hypersurface of dimension $n-1$.

(3) If $n$ is odd, then any of the $r$ liftings of the vanishing sphere $S^n_{\epsilon} \subset X_{\epsilon^2}$ to $\tilde{X}$ vanishes in $\HH_n(\tilde{X},\Z)$.

(4) If $n$ is even and $r$ is odd, then any of the $r$ liftings of the vanishing sphere $S^n_{\epsilon} \subset X_{\epsilon^2}$ to $\tilde{X}$ vanishes in $\HH_n(\tilde{X},\Z)$.

(5) If $n$ is even and $r$ is also even, then any of the $r$ the liftings of the vanishing sphere $S^n_{\epsilon} \subset X_{\epsilon^2}$ to $\tilde{X}$ is homologous to some sphere $S^n\subset Q$. Furthermore, the sphere $S^n$ vanishes in $\HH_n(\PP^{n+1},\Z)$ under the embedding $Q\hookrightarrow \PP^{n+1}$ of $Q$ as a quadric hypersurface.
\end{lem}
\begin{proof}
Most of the statements are direct application of Lemma \ref{lem induction}. We only need to prove (3) when $r$ is even. In this case, by (1) of Lemma \ref{lem induction}, we know that the lifting of $S^n_{\epsilon}\subset X_{\epsilon^2}$ to $\tilde{{X}}$ is homologous to an $n$-sphere $S^n\subset Q$. Thus the homology class of the lifting of $S^n_\epsilon$ lands in the image of 
\[
\HH_n(Q,\Z) \longrightarrow \HH_n(\tilde{X},\Z).
\]
When $n$ is odd, we have $\HH_n(Q,\Z) = 0$ since a smooth quadric has trivial homology goup in odd degree. Thus we obtain the vanishing in (3).
\end{proof}

\subsection{Global situation}

Let $\mathcal{X}$ be a smooth algeraic variety of dimension $n+1$ and $B$ a smooth curve. Let $\pi: \mathcal{X} \longrightarrow B$ be a proper morphism such that the following conditions holds.
\begin{itemize}
\item There exists a set $S=\{b_1,b_2, \ldots, b_m\}\subset B$ of finitely many points such that $X_{b_i}=\pi^{-1}b_i$ contains exactly one isolated singular point $P_i$ which is an ordinary double point.
\item The morphism $\pi$ is smooth over $B\backslash S$.
\end{itemize}
Let $0\in B$ be a point not in $S$ and let $X=X_0$. Thus $X$ is a smooth complete variety over $\C$. Let $\Delta_i\subset B$ be a small disc centered at $b_i$. Let $\Delta_i^* := \Delta_i\backslash\{b_i\}$. 
\begin{defn}
A sphere $S_\epsilon^n\subset X_{t_i}$, $t_i\in\Delta_i^*$, is called a \textit{vanishing sphere} associated to $P_i$ if the following conditions hold: (1) there exist local coordinates $\mathbf{z}=(z_0,z_1,\ldots,z_n)$ of $\mathcal{X}$ at $P_i=(0,0,\ldots,0)$;  (2) there is a local coordinate $t$ on $\Delta_i$ such that $\pi$ is locally given by 
$$
t=\pi (\mathbf{z}) = z_0^2+ z_1^2 + \cdots + z_n^2;
$$
(3) with the above coordinates, we have $t_i=\epsilon^2$ and $S^n_{\epsilon}$ is given by all points $\mathbf{z} = (x_0, x_1,\ldots,x_n)$ with $x_i\in\R$ and $\sum x_i^2=\epsilon^2$.
\end{defn}

Let $\gamma: [0,1] \longrightarrow B\backslash S$ be a continous path such that $\gamma(0)=0$ and $\gamma(1) = t_i\in \Delta_i^*$ for some $i$. Then
\[
\gamma_* : \HH_n(X,\Z) \longrightarrow \HH_n(X_{t_i},\Z)
\]
is an isomorphism. 
\begin{defn}
We say that a class $\alpha\in \HH_n(X,\Z)$ is a \textit{primitive vanishing class} if there exists a path $\gamma:[0,1]\longrightarrow B\backslash S$ as above such that $\gamma_*\alpha\in \HH_n(X_{t_i},\Z)$ is the class of a vanishing sphere associated to $P_i$. A class $\alpha'\in \HH_n(X,\Z)$ is a \textit{vanishing class} (associated to $\mathcal{X}/B$) if it is an integral linear combination of primitive vanishing classes.
\end{defn}

\begin{prop}\label{prop global vc}
Let $\mathcal{X}$ be a smooth algeraic variety and $B$ a smooth curve. Let $\pi: \mathcal{X} \longrightarrow B$ be a proper morphism as above. Let $0\in B\backslash S$ and $X= X_0$. Let $\tilde{B}$ be another smooth curve and let $f: \tilde{B}\rightarrow B$ be a finite morphism. Let $\mathcal{X}':=\mathcal{X}\times_B \tilde{B}$ be the base change of $\mathcal{X}$ and let $\tilde{\mathcal{X}}$ be a resolution of $\mathcal{X}'$. Let $\tilde{\pi}: \tilde{\mathcal{X}}\longrightarrow \tilde{B}$ be the resulting morphism induced from the morphism $\pi$. Let $\tilde{0}\in \tilde{B}$, such that $f(\tilde{0})=0$ and hence $X\cong \tilde{\pi}^{-1}\tilde{0}$. Let $j: X\hookrightarrow \tilde{\mathcal{X}}$ be the embedding. Let $\alpha\in \HH_n(X,\Z)$ be a vanishing class associated to $\mathcal{X}/B$.

(1) If $n$ is odd, then $j_*\alpha = 0$ in $\HH_n(\tilde{\mathcal{X}},\Z)$.

(2) If $n$ is even and $\tilde{\mathcal{X}}$ is obtained by successively blowing up the singular points, then $j_*\alpha$ is in the image of 
\[
 \bigoplus_{l=1}^N \HH_n(Q_{l},\Z)_{\mathrm{van}}\longrightarrow \HH_n(\tilde{\mathcal{X}},\Z),
 \] 
 where $Q_l$ runs through all smooth qudric hypersufaces appearing as components of the exceptional set of the morphism $\tilde{\mathcal{X}}\longrightarrow \mathcal{X}'$ and $\HH_n(Q_l,\Z)_{\mathrm{van}}$ consists of classes $\beta\in \HH_n(Q_l,\Z)$ that vanish in $\HH_n(\PP^{n+1},\Z)$ under the natural embedding $Q_l\subset \PP^{n+1}$.
\end{prop}

\begin{proof}
We first look at the local behaviour of the morphism $f: \tilde{B}\longrightarrow B$ around a point $b_i\in S$. Assume that
\[
 f^{-1}b_i = \{b'_{i,1}, b'_{i,2}, \ldots, b'_{i,m_i}\}\subset \tilde{B}.
 \] 
 For each point $b'_{i,l}$, we can find a small disc $\Delta_{i,l}\subset \tilde{B}$ centered at $b'_{i,l}$ such that the morphism $f:\tilde{B}\longrightarrow B$ restricts to the analytic map
 \[
 f_{i,l}: \Delta_{i,l}\longrightarrow \Delta_i,\qquad z\mapsto z^{r_{i,l}}.
 \]
To prove the proposition, we first assume that the resolution $\tilde{\mathcal{X}}\longrightarrow \mathcal{X}'$ is the one obtained by successively blowing up the singular points. Without loss of generality, we may assume that $\alpha$ is a primitive vanishing class. Thus there is a path $\gamma: [0,1]\longrightarrow B\backslash S$ with $\gamma(0) = 0\in B$ and $\gamma(1)=t_i \in \Delta_i^*$ such that $\gamma_* \alpha$ is the class of a vanishing sphere in $X_{t_i}$. We may choose $\gamma$ in such a way that it avoids all the branching points of the morphism $f$. Thus there exists a unique lifting $\tilde{\gamma}: [0,1]\longrightarrow \tilde{B}$ such that $\tilde{\gamma}(0) = \tilde{0}$. Then $\tilde\gamma(1)=\tilde{t}_i\in \Delta_{i,l}^*$ for some $l\in \{1,2,\ldots,m_i\}$. Furthermore, $\tilde\gamma_*\alpha$ is the class of a lifting $\tilde{S}^n_\epsilon$ of the vanishing sphere $S^n_{\epsilon}$ in $X_{t_i}$. If $n$ is odd, then by Lemma \ref{lem local vc} (3) we know that the homology class of $\tilde{S}^n_\epsilon$ vanishes in $\HH_n (\tilde\pi^{-1}\Delta_{i,l}, \Z)$ and hence also in $\HH_n(\tilde{\mathcal{X}},\Z)$. Similarly, if $n$ is even, we conclude from Lemma \ref{lem local vc} (4) and (5).

Now assume $n$ is odd. We still need to establish the vanishing on an arbitrary resolotion $\tilde{\mathcal{X}}_1$ of $\mathcal{X}'$. Let $\tilde{\mathcal{X}}$ be the resolution of $\mathcal{X}'$ obtained by successively blowing up the singular points. Then we can find another resolution $\tilde{\mathcal{X}}_2$ which dominates both $\tilde{\mathcal{X}}$ and $\tilde{\mathcal{X}}_1$, namely we have a diagram
\[
\xymatrix{
 &\tilde{\mathcal{X}}_2\ar[ld]_{\tau}\ar[rd]^{\tau'} &\\
 \tilde{\mathcal{X}} &&\tilde{\mathcal{X}}_1
}
\]
Let $j_1: X\hookrightarrow \tilde{\mathcal{X}}_1$, $j_2: X\hookrightarrow \tilde{\mathcal{X}}_2$ and $j: X\hookrightarrow \tilde{\mathcal{X}}$ be the inclusion of the fiber over $\tilde{0}\in \tilde{B}$ in the corresponding models. Set $\alpha_1=j_{1,*}\alpha$, $\alpha_2=j_{2,*}\alpha$ and $\tilde{\alpha}=j_*\alpha$ to be the corresponding homology classes. We have already see that $\tilde\alpha=0$. Since these models are isomorhic on an open neighborhood of the fiber $X$. We have
$\alpha_2 = \tau^*\tilde{\alpha}=0$ and $\alpha_1 = \tau'_*\alpha_2=0$. 
\end{proof}

\section{Applications to the integral Hodge conjecture}

In this section, we construct a class of new examples of the failure of the integral Hodge conjecture. These generalises the examples of Benoist--Ottem \cite{bo}. 

Let $S$ be an Enriques surface. The cohomology groups of $S$ are described as follows.
\begin{align*}
\HH^0(S,\Z) = \Z,\qquad &\HH^0(S,\Z/2\Z) = \Z/2\Z,\\
\HH^1(S,\Z) = 0,\qquad &\HH^1(S,\Z/2\Z) =\Z/2\Z, \\
\HH^2(S,\Z) = \Z^{\oplus 10}\oplus \Z/2\Z,\qquad &\HH^2(S,\Z/2\Z) = (\Z/2\Z)^{\oplus 12} ,\\
\HH^3(S,\Z) = \Z/2\Z,\qquad &\HH^3(S,\Z/2\Z) = \Z/2\Z,\\
\HH^4(S,\Z) =\Z,\qquad &\HH^4(S,\Z/2\Z) = \Z/2\Z.
\end{align*}

\subsection{Special case: hypersurfaces}

Let $X\subset \PP^{n+1}$ be a smooth hypersurface. Assume that the dimension $n$ of $X$ is odd. By Lefschetz Hyperplane Theorem, we know that 
\[
\HH^p(\PP^{n+1},\Z) \longrightarrow \HH^p(X,\Z)
\]
is an isomorphism for $p<n$. Thus $\HH^{<n}(X,\Z)$ is torsion free and algebraic. Similarly, $\HH_{<n}(X,\Z)$ is torsion free. By Serre duality, we conclude that $\HH^{>n}(X,\Z)$ is also torsion free and so is $\HH_{>n}(X,\Z)$. Then by the universal coefficient thoerem for cohomology, we see that
\[
\HH^n(X,\Z) \cong \Hom_{\Z}(\HH_n(X,\Z), \Z)\; \oplus \;\HH_{n-1}(X,\Z)_{\mathrm{tor}}
\]
is also torsion free. Hence we conclude that both $\HH^*(X,\Z)$ and $\HH_*(X,\Z)$ are torsion free.

\begin{lem}
The following equality holds
\[
\HH^{n+3}(X\times S,\Z) = \bigoplus_{i=0}^4 \HH^{n+3-i}(X,\Z)\otimes\HH^i (S,\Z).
\]
\end{lem}
\begin{proof}
The K\"unneth formula applied to this case gives
\[
 \HH^{n+3}(X\times S,\Z)=
\left( \bigoplus_{i=0}^4 \HH^{n+3-i}(X,\Z)\otimes\HH^i (S,\Z)\right) \oplus \left( \bigoplus_{i=0}^4 \mathrm{Tor}_1\Big( \HH^{n+4-i}(X,\Z), \HH^i (S,\Z) \Big)\right)
\]
Since the cohomology of $X$ is torsion free, we see that in the $\mathrm{Tor}_1$-term vanishes.
\end{proof}

\begin{thm}\label{thm non-ext}
Let $\pi: \mathcal{X} \longrightarrow B=\PP^1$ be a Lefschetz pencil of smooth hypersurfaces of odd dimension $n$. Let $X=\pi^{-1}0$ be a smooth fiber. Then any non-zero element $\alpha \in \HH^n(X,R)$ is non-extendable, where $R$ is a non-zero commutative ring.
\end{thm}

\begin{proof}
Assume that $\alpha$ is extendable. Then there exists a smooth projective curve $\tilde{B}$ and a finite morphism $\tilde{B}\longrightarrow B$ such that a resolution $\tilde{\mathcal{X}}$ of the base change $\mathcal{X}'=\mathcal{X}\times_{B}\tilde{B}$ is obtained by successively blowing up the singular points. Let $\tilde{\pi}:\tilde{\mathcal{X}}\longrightarrow \tilde{B}$ be the induced morphism. Furthermore, we have a cohomology class $\tilde{\alpha} \in \HH^n(\tilde{\mathcal{X}}, R)$ such that $\tilde{\alpha}|_X = \alpha$, where $X=\tilde{\pi}^{-1}\tilde{0}$ for some preimage $\tilde{0}\in \tilde{B}$ of $0$. Let $j: X \hookrightarrow \tilde{\mathcal{X}}$ be the inclusion. Let $\beta\in\HH_n(X,\Z)$. Since $n$ is odd, we know that $\beta$ vahishes in $\HH_n(\PP^{n+1},\Z)$. By Lefschetz theory (see for example \cite{lam}), we know that $\beta$ is a vanishing class associated to the Lefschetz pencil $\mathcal{X}\longrightarrow B$. Then by (1) of Proposition \ref{prop global vc}, we see that $j_*\beta =0$ in $\HH_n(\tilde{\mathcal{X}},\Z)$. Thus 
\[
\langle\alpha,\beta\rangle = \langle j^*\tilde\alpha,\beta \rangle = \langle \tilde{\alpha}, j_*\beta \rangle = 0.
\]
This forces that $\alpha = 0$ since $\HH^n(X,R) = \Hom_{\Z}(\HH_n(X,\Z), R)$.
\end{proof}

\begin{cor}\label{cor non-alg}
Let $i: X\hookrightarrow \PP^{n+1}$ be a smooth hypersurface of odd dimension $n$ and let $S$ be an Enriques surface. Let $\alpha\in \HH^n(X,\Z)$ be an element not divisible by $2$ and let $u\in \HH^3(S,\Z)$ be the unique nonzero element. Then the torsion class $\alpha\otimes u\in \HH^{n+3}(X\times S,\Z)$ is non-extendable in $\mathcal{X}\times S \longrightarrow B$ for any Lefschetz pencil $\mathcal{X}\longrightarrow B$ containing $X$. If $X$ is very genery, then $\alpha\otimes u$ is not algebraic.
\end{cor}

\begin{proof}
Let $\pi: \mathcal{X}\longrightarrow B=\PP^1$ be a Lefschetz pencil of hypersurfaces of dimension $n$ such that for some point $0\in B$ the corresponding fiber $X_0:=\pi^{-1}0\cong X$. Assume that $\alpha\otimes u$ is extendable. As in the above proof, there exist a smooth projective curve $\tilde{B}$, a fninite morphism $\tilde{B}\rightarrow B$, a resolution $\tilde{\mathcal{X}}$ of $\mathcal{X}':=\mathcal{X}\times_B \tilde{B}$, an identification $X=\tilde{\pi}^{-1}(\tilde{0})$ and a cohomology class $\tau \in \HH^{n+3}(\tilde{\mathcal{X}}\times S,\Z)$ such that $\tau|_{X\times S} = \alpha\otimes u$. Consider the class $\bar{\tau}\in \HH^{n+3}(\tilde{\mathcal{X}}\times S,\Z/2\Z)$ which is obtained from $\tau$ modulo 2. Let $u'\in \HH^1(S,\Z/2\Z)$ be the unique non-zero element which is associated to the $K3$ covering $\tilde{S}\longrightarrow S$. Then we get
\[
 \tilde{\alpha} := \bar{\tau}^* u' \in \HH^n(\tilde{\mathcal{X}},\Z/2\Z)
\]
which stisfies the following condition
\[
\tilde{\alpha}|_{X} = (\bar{\tau}^* u')|_{X} = (\bar{\tau}|_{X\times S})^*u' = (\bar\alpha \otimes \bar{u})^*u' = \bar\alpha
\]
where $\bar\alpha$ is the image of $\alpha$ in $\HH^n(X,\Z/2\Z)$ and $\bar{u}$ is the image of $u$ in $\HH^3(S,\Z/2\Z)$. The last equality uses the duality relation $\langle \bar{u}, u'\rangle =1$. It follows that $\bar{\alpha}\in \HH^n(X,\Z/2\Z)$ is extendable. By the above theorem, we have $\bar{\alpha} = 0$ and hence $\alpha$ is divisible by $2$ in $\HH^n(X,\Z)$. This give a contradiction.

It follows that $\alpha\otimes u$ is not algebraic for a very general memeber $X$ in a Lefschetz pencil. In particular, this holds for a very general $X$.
\end{proof}

\subsection{General case: hyperplane sections}
Let $Y$ be a smooth projective variety with a very ample line bundle $\calO_Y(1)$ which gives rise to an embedding $Y\hookrightarrow \PP^N$. The same argument as above gives the following.

\begin{thm}\label{thm general case}
Let $\pi:\mathcal{X}\rightarrow B=\PP^1$ be a Lefschetz pencil in $|\calO_Y(1)|$. Let $X=\pi^{-1}0$ be a smooth fiber and let $i: X\hookrightarrow Y$ be the embedding. Assume that $\dim Y = n+1$ where $n$ is an odd integer. Let $R$ be a nonzero commutative ring.

(1) If $\alpha\in \HH^n(X,R)$ is extendable, then we have
\[
\langle \alpha,\beta \rangle = 0,
\]
for all $\beta\in \HH_n(X,\Z)_{\mathrm{van}}:=\ker \{i_*: \HH_n(X,\Z)\longrightarrow \HH_n(Y,\Z)\}$. Furthermore, if $\HH_{n}(Y,\Z)$ vanishes and $\HH_{n-1}(Y,\Z)$ is torsion-free, then every nonzero element $\alpha \in \HH^n(X,R)$ is non-extendable.

(2) Let $S$ be an Enriques surface and $u\in \HH^3(S,\Z)$ be the unique nonzero element. Let $\alpha \in \HH^n(X,\Z)$. If $\alpha\otimes u$, viewed as an element in $\HH^{n+3}(X\times S,\Z)$, is extendable (in the family $\mathcal{X}\times S\rightarrow B$), then $\bar\alpha\in \HH^n(X,\Z/2\Z)$ is extendable.

(3) Assume that $\HH_n(Y,\Z)$ vanishes and that $\HH_{n-1}(Y,\Z)$ is torsion-free. If $X$ is very general in $|\calO_Y(1)|$, then for all $\alpha\in \HH^n(X,\Z)$ the class $\alpha\otimes u \in \HH^{n+3}(X,\Z)$ is not algebraic unless it is zero.
\end{thm}

\begin{proof}
We will use the notations $\tilde{B}$, $\mathcal{X}'$, $\tilde{\mathcal{X}}$, $j: X=\tilde{\pi}^{-1}\tilde{0} \hookrightarrow \tilde{\mathcal{X}}$ as in the previous proofs. 

(1) If $\alpha$ is extendable, then there exists $\tilde{\alpha}\in \HH^n(\tilde{\mathcal{X}}, R)$ such that $\alpha = j^*\tilde{\alpha}$. Then we again have
\[
\langle \alpha,\beta\rangle = \langle j^*\tilde\alpha,\beta\rangle = \langle \alpha,j_*\beta\rangle =0,
\]
since $j_*\beta = 0$ for all $\beta\in \HH_n(X,\Z)_{\mathrm{van}}$ by Lefschetz theory and (1) of Proposition \ref{prop global vc}. Assume that $\HH_n(Y,\Z) = 0$, then we have
\[
\HH^n(X,\Z)_{\mathrm{van}} = \HH_n(X,\Z).
\]
If $\HH_{n-1}(Y,\Z)$ is torsion free, then by Lefschetz hyperplane theorem, we know that $\HH_{n-1}(X,\Z)$ is also torsion free. Then the universal coefficient theorem for cohomology becomes
\[
\HH^n(X,R) = \Hom_{\Z}(\HH_n(X,\Z), R).
\]
Thus the vanishing of $\langle\alpha,\beta\rangle = 0$ for all $\beta \in \HH_n(X,\Z)$ implies $\alpha = 0$ in $\HH^n(X,R)$.

(2) and (3): the proof is the same as that of Corollary \ref{cor non-alg}. One only needs to note that ,under the assumptions of (3), the group $\HH^{n+1}(X,\Z)$ is also torsion free by Poincar\'e duality. Thus the universal coefficient theorem implies
\[
 \HH^n(X,\Z/2\Z) = \HH^n(X,\Z)\otimes \Z/2\Z.
\]
Thus $\bar\alpha =0 $ in $\HH^n(X,\Z/2\Z)$ if and only if $\alpha\otimes u =0$ in $\HH^{n+3}(X\times S,\Z)$ since the K\"unneth formula gives
\[
 \HH^n(X,\Z)\otimes \HH^(S,\Z) \hookrightarrow \HH^{n+3}(X\times S,\Z)
\]
and $\HH^3(S,\Z) = \Z/2\Z$. Then (3) follows from (1) and (2)
\end{proof}

\end{document}